\documentclass[preprint,12pt]{elsarticle}

\usepackage{amssymb}
\usepackage{amsmath}
\usepackage{amsthm} 
\newtheorem{theorem}{Theorem}
\newtheorem{definition}[theorem]{Definition}
\usepackage{algorithm}
\usepackage{algpseudocode} 
\usepackage{hyperref}
\usepackage{caption}
\usepackage{subcaption}
\journal{Arxiv}

\begin{document}

\begin{frontmatter}


%\author{Haoran Zhang\corref{cor1}\fnref{label2}}
%\ead{mike-haoran.zhang@connect.polyu.hk}

\title{Data-Driven Long-Term Asset Allocation with Tsallis Entropy Regularization} %% Article title

 %use optional labels to link authors explicitly to addresses:
\author[label1]{Haoran Zhang}
\affiliation[label1]{organization={Department of Applied Mathematics, The Hong Kong Polytechnic University},
             country={Hong Kong}}
             
\author[label2]{Wenhao Zhang}
\affiliation[label2]{organization={Department of Applied Mathematics, The Hong Kong Polytechnic University},
             country={Hong Kong}}
 
\author[label3]{Xianping Wu}            
\affiliation[label3]{organization={School of School of Mathematics and Statistics, Guangdong University of Technology},
             %addressline={},
             city={GuangZhou},
             postcode={510520},
             state={Guangdong},
             country={China}}

%\author[label1, label2, label3]{Haoran Zhang, Wenhao Zhang, Xianping Wu} %% Author name

%Author affiliation
%\affiliation{organization={Department of Applied Mathematics, The Hong Kong Polytechnic University},%Department and Organization
            %addressline={}, 
   %         city={Hong Kong},
            %postcode={}, 
            %state={},
      %      country={China}}

%% Abstract
\begin{abstract}
This paper addresses the problem of dynamic asset allocation under uncertainty, which can be formulated as a linear–quadratic (LQ) control problem with multiplicative noise. To handle exploration–exploitation trade-offs and induce sparse control actions, we introduce Tsallis entropy as a regularization term. We develop an entropy-regularized policy iteration scheme and provide theoretical guarantees for its convergence. For cases where system dynamics are unknown, we further propose a fully data-driven algorithm that estimates Q-functions using an instrumental-variable least-squares approach, allowing efficient and stable policy updates. Our framework connects entropy-regularized stochastic control with model-free reinforcement learning, offering new tools for intelligent decision-making in finance and automation.
\end{abstract}


\end{frontmatter}



\section{Introduction}
The Linear-Quadratic (LQ) problem constitutes a foundational pillar of control theory. The primary aim is to identify an optimal control strategy identification for linear dynamics, whose corresponding cost function is minimum \cite{abeille2016lqg}. The optimal control input can be expressed functions as a feedback with respect to state, which comes from the algebraic Riccati equation (ARE). The LQ problem is unique among optimal control in that it has closed-form analytical feedback solution \cite{almgren2001optimal}. When the dynamics exhibit nonlinearity and are difficult to analyze, LQ approximation will be operated with local linearization.. This approximation is close to the original problem near state equilibrium \cite{benigno2012linear}. In recent years, the LQ problem has attracted substantial interest due to its versatile applications among multiple domains, including dynamic portfolio management and optimization \cite{li2002dynamic, cui2014unified, forouzanfar2024integrated}, financial derivative pricing \cite{primbs2009stochastic}, transfer learning \cite{cao2023feasibility}, resource extraction \cite{yang2015linear}, biomedical engineering \cite{chavez2005linear} and optimal liquidation \cite{almgren2001optimal}. 

However, despite its analytical elegance, the LQ framework presents several challenges in practice. First, ARE involves nonlinearity and is impossible to be directly solved unless it iterates with proper initialization, still less high-dimensional or time-varying systems. Second, classical LQ control assumes perfect knowledge of the system parameters. In real-world scenarios where the dynamics are unknown or noisy, this assumption is often violated. Model identification techniques may offer partial remedies, but they introduce additional uncertainty and reduce robustness—particularly in online or adaptive settings. These limitations underscore the need for data-driven approaches that can bypass the dependency on explicit models while still solving dynamic decision problems efficiently.

Reinforcement Learning (RL) is a core paradigm in machine learning that naturally addresses these limitations. RL focuses on how agents take actions in an environment to maximize cumulative rewards over time \cite{sutton2018reinforcement}. RL is distinguished by its unique ability to address sequential decision-making problems, which  system dynamics are often partially known or completely unknown. The framework of RL typically revolves around the Markov Decision Process (MDP), comprising states, actions, rewards, and a transition function, which collectively define the agent-environment interaction \cite{puterman2014markov}. A key advantage of RL is its flexibility: agents learn optimal policies through trial-and-error interactions, rather than explicit system models. This adaptability has enabled RL applications in diverse fields, from robotics and game playing to finance and control systems \cite{levine2020offline}.

Traditional LQ control yields deterministic policies, which can lead to poor exploration and instability in uncertain or data-driven environments. In RL, this limitation often causes premature policy convergence and inefficient learning. Entropy regularization mitigates this by encouraging exploration and smoothing policy updates, improving both robustness and convergence behavior \cite{cen2022fast}. In contrast to Shannon entropy, which yields softmax-type policies, Tsallis entropy generalizes Shannon entropy through a parameter $q \in \mathbb{R}$, which allows one to flexibly control the trade-off between sparsity and randomness in the learned policy. When $q=1$, Tsallis entropy reduces to the Shannon entropy; for $q<1$, it encourages sparsity, while $q>1$ results in more diffused stochasticity. This adaptability makes it especially suitable for control tasks under multiplicative noise, where a balance between exploration and structure is crucial. \\

\textbf{Our Contribution.} This paper investigates the integration of Tsallis entropy into the LQ control framework with multiplicative noise, and proposes a novel policy iteration method grounded in Q-functions. Tsallis entropy, as a one-parameter generalization of Shannon entropy, serves as a regularizer that balances exploration and control sparsity. We derive explicit solutions under this regularization, demonstrating improved robustness and adaptability in stochastic systems subject to multiplicative uncertainty.

For completely unknown dynamics, we develop a fully data-driven policy iteration algorithm tailored to LQ systems with multiplicative noise. Our approach estimates Q-functions through a least-squares formulation with instrumental variables, following principles akin to least-squares temporal difference (LSTD) learning. By incorporating Tsallis entropy into this reinforcement learning framework, our method achieves stable and efficient policy updates without requiring knowledge of system parameters.

The proposed framework closes an important gap in existing literature by extending entropy-regularized control to data-driven LQ settings with multiplicative noise. It provides a theoretically grounded and practically effective approach to entropy-based stochastic control in complex, uncertain environments.

\subsection{Related Work}
\noindent \textbf{LQ Formulation for the Mean-Variance Problem.} 
The exploration of LQ stochastic optimal control problems, especially those with multiplicative noise, has provided a robust framework for addressing complex decision-making scenarios in systems engineering and financial mathematics in \cite{cui2024limited}. One of key to this kernel is the application of LQ problem to dynamic Mean-Variance (MV) portfolio analysis, a sophisticated extension of Markowitz's classical static portfolio selection theory \cite{zhou2000continuous, cui2014unified}. These models, despiting the challenges posed by indefinite penalty matrices on control and state variables, have been shown to be well-posed under certain conditions, thereby offering a promising avenue for optimizing portfolio selections in a stochastic setting. Several researches have investigated the issues of estimating and controlling systems that have multiplicative noise \cite{basin2006optimal, wu2018explicit}. Research has mostly concentrated on the LQ formulation with indefinite penalty matrices for infinite horizon discrete-time models in stochastic optimal control problems with multiplicative noise of the LQ type \cite{mukaidani2017infinite}. It is worth mentioning that  models with multiplicative noise can still be considered well-posed under specific conditions even when the penalty matrices for both the state and control are not clearly defined. \\

\noindent \textbf{Reinforcement Learning in LQ Problem.} Within RL, two primary methodologies dominate: model-based methods, which attempt to explicitly estimate the underlying transition dynamics, and model-free methods, which directly learn policies or value functions without requiring transition kernel estimation \cite{jin2020analysis}. While model-based methods can converge faster given accurate models, their dependence on accurate system knowledge limits their applicability in uncertain or complex environments. Conversely, model-free methods, such as policy gradient and Q-learning algorithms, have gained prominence for their robustness in environments where the system dynamics are opaque or highly variable \cite{haarnoja2018soft, fan2024q}.

By integrating RL into the LQ framework, agents can learn optimal control policies directly from data. Model-free RL methods, in particular, have demonstrated considerable promise for solving LQ problems without requiring precise parameter knowledge. Methods such as policy gradient \cite{fazel2018global} and least-squares policy iteration \cite{hambly2021policy} have been adapted to discrete-time LQ settings, which achieves convergence guarantees under various assumptions. These methods leverage the trial-and-error nature of RL to iteratively refine control policies, which also address challenges such as state‑multiplicative noise and entropy regularization. The synergy between RL and LQ problem opens new avenues for solving practical control problems in engineering, robotics, and financial systems, whose analytical solutions are either infeasible or suboptimal in the presence of real-world complexities.\\ 

\noindent \textbf{Entropy Regularization.} The entropy regularization utilization is prevalent in the field of RL, as evidenced in the literature by \cite{sutton2018reinforcement, miao2024effective}. Entropy regularization is frequently employed to encourage exploration and improve convergence. Research has shown that incorporating entropy regularization leads to a smoother and more consistent landscape, allowing for the use of higher learning rates. As a result, this results in faster convergence during the training phase. The notion of entropy regularization is crucial for achieving a balance between exploration and exploitation in RL. This concept has been highlighted in the works of \cite{haarnoja2018soft}. Sometimes, the reward function may incorporate Tsallis entropy, as introduced by \cite{tsallis1988possible}, to represent the agent's action distribution. This approach generalizes Shannon entropy for non-extensive systems, where standard additivity is not applicable, including large-scale interaction, fractal structures, or ergodic dynamics. This generalization alters the optimal action distribution to be q-Gaussian rather than Gaussian, as established by the Shannon entropy in \cite{wang2020reinforcement}.  In \cite{bao2022sparse}, Tsallis entropy, derived from Tsallis statistical mechanics \cite{tsallis1988possible}, is employed to regularize optimal transport problems, yielding high-entropy yet sparse solutions. In RL, \cite{lee2018sparse} propose a method for obtaining sparse control policies through a specific application of Tsallis entropy.

\subsection{Notations and Organization}
In this paper, we adopt the following standard mathematical notations. Let $\mathbb{R}$ denote the field of real numbers, and $\mathbb{R}^+$ its subset of nonnegative elements. For a matrix $Z$, $Z^\top$ denotes its transpose; $|Z|$ and $|Z|_F$ represent its spectral norm and Frobenius norm, respectively. The trace of a matrix is denoted by $\operatorname{Tr}(Z)$. The operator $\operatorname{vec}_{\mathrm{s}}(Z)$ refers to the symmetrized vectorization of a symmetric matrix $Z$, while $\operatorname{unvec}_{\mathrm{s}}$ denotes its inverse, mapping the vector back to a symmetric matrix. 

The structure of the rest of paper is as follows. Section 2 introduces the formulation of the infinite-horizon LQ control problem with multiplicative noise, and derives its optimal solution under Tsallis entropy regularization. Section 3 presents the Q-function-based dynamic programming principle tailored for the Tsallis entropy setting and establishes the convergence of the proposed policy iteration algorithm. Section 4 develops a data-driven control framework by constructing a model-free policy iteration algorithm, which estimates Q-functions using instrumental variable regression, extending least-squares temporal difference methods. Section 5 provides numerical simulations in the context of the MV portfolio optimization problem, demonstrating the effectiveness of the proposed approach. 

\section{Preliminary and Problem Formulation}
This section provides a concise introduction to the essential mathematical principles of basic LQ problem with multiplicative noise and the Tsallis entropy preliminary. The typical real life example, i.e., MV problem, along with a compilation of requisite results utilized throughout the paper, are discussed in the later section. 

\subsection{Tsallis Entropy}
To enhance exploration throughout the control selection process, we modify the performance criterion by integrating a reward based on entropy. We employ Tsallis entropy as a measure to quantify the degree of uncertainty in the control distribution. This concept was initially presented by \cite{tsallis1988possible}. Before presenting the Tsallis entropy $\mathcal{T}_q$, it is essential to construct the $q$-exponential and $q$-logarithm functions, as well as the multivariate $q$-Gaussian within the framework of Tsallis statistics \cite{hashizume2024tsallis}. The variable $q$ is designated as a deformation parameter \cite{lavagno2002non}. As $q \rightarrow 1$, the formula reduces to the Shannon entropy. In this study, we assume that the value of $q \in (0,1)$ for simplicity. This assumption is particularly suitable for capturing heavy-tailed risks and ensuring robust allocation decisions in financial settings.

\begin{definition}[$q$-Exponential and $q$-Logarithm functions]
\begin{equation}
\begin{aligned}
\exp _q(x)&=[1+(1-q) x]_{+}^{\frac{1}{1-q}} & & x \in \mathbb{R}, \\
\log _q(x)&=\frac{x^{1-q}-1}{1-q} & & x>0,
\end{aligned}
\label{eqTE2.1}
\end{equation}
where $[x]_{+}=\max (x, 0)$. These functions have the inverse property that
\begin{equation}
\begin{aligned}
\exp _q\left(\log _q(x)\right)&=x \quad x>0, \\
\log _q\left(\exp _q(x)\right)&=x \quad x>-\frac{1}{1-q}.
\end{aligned}
\label{eqTE2.2}
\end{equation}
\label{def2.1}
\end{definition}

\begin{definition}[Multivariate $q$-Gaussian]
A $q$-Gaussian $N_q(\mu, \Sigma)$ is a $n$-dimensional random variable characterized by its density function as follows:
\begin{equation}
\varphi(x)=\frac{1}{Z_q} \exp _q\left(-\frac{(x-\mu)^{\top} \Sigma^{-1}(x-\mu)}{(n+4)-(n+2) q}\right)
\label{eqTE2.3}
\end{equation}
where
$$
Z_q=\operatorname{det}(\Sigma)^{\frac{1}{2}}\left(\pi \frac{(n+4)-(n+2) q}{1-q}\right)^{\frac{n}{2}} \frac{\Gamma\left(\frac{2-q}{1-q}\right)}{\Gamma\left(\frac{2-q}{1-q}+\frac{n}{2}\right)}.
$$
\label{def2.2}
\end{definition}
Subsequently, we define Tsallis entropy, which functions as an extension of Shannon entropy:
\begin{definition}[Tsallis entropy and $q$-Entropy]
Tsallis entropy $\mathcal{T}_q$ is defined as
\begin{equation}
\mathcal{T}_q(\varphi)=-\int \varphi(x)^q \log _q \varphi(x) d x.
\label{eqTE2.4}
\end{equation}
The deformed $q$-entropy employs a regularization term, as detailed below \cite{bao2022sparse}. The subsequent relationship, termed additive duality, ties it with Tsallis entropy. The deformed Tsallis entropy, or $q$-entropy $\mathcal{H}_q$ is defined as
\begin{equation}
\mathcal{H}_q(\varphi)=-\frac{1}{2-q}\left(\int \varphi(x) \log _q \varphi(x) d x-1\right).
\label{eqTE2.5}
\end{equation} 
In the following section, we will use the deformed q-entropy for the sake of notational simplicity. 
\label{def2.3}
\end{definition}

\subsection{LQ Problem} \label{LQmodel}
Define the admissible policy set as $\Pi=\{\pi: \mathcal{X} \rightarrow \mathcal{P}(\mathcal{U})\}$, with $\mathcal{X}$ the state space, $\mathcal{U}$ the action space, and $\mathcal{P}(\mathcal{U})$ the space of probability measures on action space $\mathcal{U}$. Here each admissible policy $\pi \in \Pi$ maps a state $x \in \mathcal{X}$ to a randomized action in $\mathcal{U}$.

According to Definition \ref{def2.3}, we regard the deformed $q$-entropy as a regularization term in this study. For a given admissible policy $\pi \in \Pi$, the corresponding Tsallis entropy is defined as for $ \forall x \in \mathcal{X}$,
\begin{equation}
\mathcal{H}_q(\pi(\cdot|x))=-\frac{1}{2-q}\left(\int_{\mathcal{U}} \pi(u | x) \log_q \pi(u | x) du-1\right), 
\label{eqH}
\end{equation}
where $-\frac{1}{2-q}\left(\mathbb {E} [\log_q \pi (u | x ) ]-1\right)$ can be used to generalize the above. 

The decision maker aims to find an optimal policy by minimizing the following objective function
\begin{equation}
\min _{\pi \in \Pi} \mathbb{E}_{x \sim \mathcal{D}}\left[J_\pi(x)\right],
\label{eqp-2.1.1}
\end{equation}
where value function $J_\pi$ for infinite horizon case is given with negative entropy:
\begin{equation}
J_\pi(x)=\mathbb{E}_\pi\left[\sum_{t=0}^{\infty} \gamma^t c\left(x_t, u_t\right) \Bigg| x_0=x\right],
\label{eqp-2.1.2}
\end{equation}
where $c\left(x_t, u_t\right) = x_t^{\top} Q x_t+u_t^{\top} R u_t + \frac{\tau}{2-q}\left(\log_q \pi (u_t | x_t) -1\right)$ such that for $t=0,1,2, \cdots$,
\begin{equation}
x_{t+1}=(A+ {\Delta A}_t) x_t+(B+ {\Delta B}_t) u_t+ (C x_t + D u_t) w_t, \ x_0 \sim \mathcal{D}.
\label{eqp-2.1.3}
\end{equation}

Here $x_t \in \mathbb{R}^m$ is the state of the system and the initial state $x_0$ follows an initial distribution $\mathcal{D}$. Then $u_t \in \mathbb{R}^n$ is the control at time $t$ following a policy $\pi$. In addition, the system is affected by three types of noise: $\{\Delta A_t\}_{t=0}^{\infty}$ and $\{\Delta B_t\}_{t=0}^{\infty}$ represent random perturbations to the system (transition) dynamics matrices $A \in \mathbb{R}^{m \times m}$ and $B \in \mathbb{R}^{m \times n}$, respectively, and $\left\{w_t\right\}_{t=0}^{\infty}$ is the process noise. All three noises are zero-mean, mutually independent, and independent and identically distributed (i.i.d) across time. We assume that all have finite second moments. That is, $\operatorname{Tr}(\Sigma_A)<\infty$ with $\Sigma_A=\mathbb{E}\left[\Delta A_t {\Delta A_t}^{\top}\right]$,  $\operatorname{Tr}(\Sigma_B)<\infty$ with $\Sigma_B=\mathbb{E}\left[ \Delta B_t {\Delta B_t}^{\top}\right]$ and $\operatorname{Tr}(W)<\infty$ with $W=\mathbb{E}\left[w_t w_t^{\top}\right]$ for any $t=0,1,2, \cdots$. $Q \in \mathbb{R}^{m \times m}$ and $R \in \mathbb{R}^{n \times n}$ are positive definite matrices that parameterize the quadratic costs. $\gamma \in (0,1)$ denotes the discount factor and $\tau$ is the regularization parameter. The expectation in (\ref{eqp-2.1.2}) is taken with respect to the control $u_t \sim \pi\left(\cdot | x_t\right)$ and system noise $w_t$ for $t \geq 0$.

\subsection{Optimal Strategy and Algebraic Riccati Equation}
The optimal value function $J^*: \mathcal{X} \rightarrow \mathbb{R}$ is defined as
\begin{equation}
J^*(x)=\min _{\pi \in \Pi} J_\pi(x).
\label{eq2.11}
\end{equation}
The following theorem provides the explicit formulation for the optimal control policy and the associated optimal value function: the optimal policy is defined as a multivariate $q$-Gaussian distribution, with the mean being linear in the state $x$ and a constant covariance matrix.

\begin{theorem}[Optimal value functions and optimal policy]
The optimal value function $J^*: \mathcal{X} \rightarrow \mathbb{R}$ in (\ref{eq2.11}) can be expressed as 
\begin{equation}
J^*(x)=x^{\top} Px+c,
\label{TE-Jstar}
\end{equation}
where $c$ is a constant and the optimal control policy $\pi^*$ is $q$-Gaussian with mean $\mu^*$ and variance $\Sigma^*$ as follows:
{\footnotesize \begin{equation}
\begin{aligned}
P &=Q+K^{* \top} R K^*+\gamma \left(A-B K^*\right)^{\top} P\left(A-B K^*\right) \\
&\quad + \gamma \left(C-DK^*\right)^{\top}WP \left(C-DK^*\right) + \gamma \Sigma_A P +  \gamma {K^*}^{\top} \Sigma_B P K^* , \\
c &=\frac{1}{1-\gamma}\left[\operatorname{Tr}\left(\Sigma^*(R+\gamma B^{\top} P B + \gamma {\Sigma_B}P + \gamma D^{\top} WP D)\right)- \frac{\tau}{2-q}\right.\\
&\quad \left. \cdot \biggl(\log_q\left(Z_q\right)  + \frac{n}{(n+4)-(n+2)q} +  \frac{n(q -1) \log_q\left(Z_q\right)}{(n+4)-(n+2)q} +1 \biggr) \right],     
\end{aligned}
\label{eq2.12}
\end{equation}}
with
\begin{equation}
\begin{aligned}
K^*&=\gamma (R+ \gamma B^{\top} P B + \gamma {\Sigma_B}P + \gamma WD^{\top} P D)^{-1} (B^{\top} P A+ WD^{\top} P C), \\
\Sigma^* &=\frac{\tau}{(n+4)-(n+2) q} (R + \gamma B^{\top}PB+ \gamma {\Sigma_B}P + \gamma WD^{\top} P D)^{-1}.\\
\end{aligned}
\label{eq2.13}
\end{equation}
For any $x \in \mathcal{X}$, the corresponding optimal policy $\pi^*$ to (\ref{eq2.11}) is a q-Gaussian policy 
\begin{equation}
\pi^*(\cdot | x)=\mathcal{N}_q\left(-K^* x, \Sigma^*\right),
\label{TE-gaussian}
\end{equation}
where $\mu^* = -K^*x$.
\label{thm2.4}
\end{theorem} 

\begin{proof}[\textbf{Proof of Theorem \ref{thm2.4}}]
The proof of (\ref{eq2.13}) and (\ref{TE-gaussian}) in Theorem \ref{thm2.4} refers partially on \cite{hashizume2024tsallis, guo2023fast}. To derive the associated optimal value function, we first calculate $\mathbb{E}_{\pi^*}\left[\log_q \left(\pi^*(u | x)\right)\right]$ in the entropy of policy $\pi^*$ at any state $x \in \mathcal{X}$:
$$
\begin{aligned}
\mathbb{E}_{\pi^*}\left[\log_q \left(\pi^*(u | x)\right)\right]&=\int_{\mathcal{U}}\pi^*(u | x)  \log_q \left(\pi^*(u | x)\right)d u \\
&=- \log_q\left(Z_q\right) - \alpha n + (1-q)\alpha n \log_q\left(Z_q\right),
\end{aligned}
$$
where $\alpha = \frac{1}{(n+4)-(n+2)q}$ and $Z_q$ is defined in Definition \ref{eqTE2.3}. The calculation follows the product rule for $q$-Logarithm with $\log _q(a \cdot b)=\log _q(a)+\log _q(b)+(1-q) \log _q(a) \log _q(b)$. Thus the entropy is 
\begin{equation}
\begin{aligned}
\mathcal{H}_q(\pi(\cdot|x)) &= -\frac{1}{2-q}\left(\mathbb {E} [\log_q \pi (u | x ) ]-1\right) \\
& = \frac{1}{2-q}\left(\log_q\left(Z_q\right) + \alpha n +  (q-1)\alpha n \log_q\left(Z_q\right) +1 \right).
\end{aligned}
\label{TE-entropy}
\end{equation}
Substitute (\ref{TE-gaussian}), (\ref{TE-entropy}) and (\ref{TE-Jstar}) into $J^*$ with dynamic programming principle:
$$
\begin{aligned}
&\quad \ J^*(x) \\
&=  x^{\top} Q x +\min _\pi \mathbb{E}_\pi\biggl\{u^{\top} R u+ \frac{\tau}{2-q}\left( \log_q (\pi(u | x))-1\right) \\
&\quad +\gamma\Bigl[\left((A+\Delta A) x + (B+ \Delta B) u+(C x + D u)w\right)^{\top} P \\
&\quad \cdot \left((A+\Delta A) x+(B+\Delta B) u+(C x + D u)w\right)+c\Bigr]\biggr\} \\
&=  x^{\top}\Bigl(Q+K^{* \top} R K^*+\gamma\left(A-B K^*\right)^{\top} P\left(A-B K^*\right) \\
&\quad + \gamma\left(C-D K^*\right)^{\top} WP \left(C-D K^*\right)+\gamma \Sigma_A P + \gamma {K^*}^{\top} \Sigma_B P K^*\Bigr) x \\
&\quad +\operatorname{Tr}\left(\Sigma^* R\right) -  \frac{\tau}{2-q}\left(\log_q\left(Z_q\right) + \alpha n +  (q -1)\alpha n \log_q\left(Z_q\right) +1 \right) \\
&\quad +\gamma\left(\operatorname{Tr}(\Sigma^* B^{\top} P B)+ \operatorname{Tr}\left(\Sigma^* \Sigma_B P\right) +\operatorname{Tr}(\Sigma^* D^{\top} W P D)+c\right). \\
\end{aligned}
$$
The equations in ($\ref{eq2.12}$) can be observed above and the proof is complete.
\end{proof}

\section{Dynamic Programming for Tsallis Entropy LQ Problem}
In this section, we develop dynamic programming algorithms for the Tsallis entropy-regularized LQ control problem, including policy iteration and value iteration methods. These algorithms aim to compute the optimal value function and control policy, and their convergence properties can be established (Theorem \ref{TE-thm-5.2} and \ref{TE-thm-5.3}). The policy iteration method alternates between policy evaluation, which computes the value function under a fixed policy, and policy improvement, which updates the policy based on the current value function. The value iteration method directly updates the value function toward optimality. The proof of convergence is referenced to \cite{lee2019tsallis}.

Following the dynamics specified in equation (\ref{eqp-2.1.3}), we use the subscript “+” to indicate progression in time: $x_{+}$ and $u_{+}$ represent the next state and control input, respectively. This notation extends naturally to $x_{++}, u_{++}$, etc., by recursive application of the same rule. Let $Z$ be expressed as a function $\mathcal{M}$ of the parameter $X$,
\begin{equation}
\begin{aligned}
Z &= \mathbb{E}\left[\begin{bmatrix}x \\ u\end{bmatrix}\cdot \begin{bmatrix}x^\top & u^\top\end{bmatrix}\right] \\
&= \begin{bmatrix}\mathbf{I}_m \\ -K\end{bmatrix}X\begin{bmatrix}\mathbf{I}_m & -K^\top\end{bmatrix} + \begin{bmatrix} \mathbf{0}_{m\times m} & \mathbf{0}_{m \times n} \\ \mathbf{0}_{n\times m} & \Sigma \end{bmatrix} \\
&= \mathcal{M}(X),
\label{eqZ}
\end{aligned}
\end{equation}
where $X = \mathbb{E}[xx^{\top}]$. To express the second moment of state value of next time $X_{+}$ in terms of $Z$ as function $\mathcal{E}(Z)$:
\begin{equation}
\begin{aligned}
X_{+} &= \mathbb{E}[x_{+}x_{+}^\top] \\
&= \begin{bmatrix}A & B\end{bmatrix}Z\begin{bmatrix}A & B\end{bmatrix}^\top+ W\begin{bmatrix}C & D\end{bmatrix}Z\begin{bmatrix}C & D\end{bmatrix}^\top \\
&\quad + \Sigma_A \begin{bmatrix}\mathbf{I}_m & \mathbf{0}_{m \times m} \end{bmatrix} Z \begin{bmatrix}\mathbf{I}_m  & \mathbf{0}_{m \times m} \end{bmatrix}^{\top} \\
& \quad + \Sigma_B \begin{bmatrix}\mathbf{0}_{n \times n} & \mathbf{I}_n \end{bmatrix} Z \begin{bmatrix}\mathbf{0}_{n \times n} & \mathbf{I}_n \end{bmatrix}^{\top},
\end{aligned}
\label{eqX}
\end{equation}
and the value function (\ref{eqp-2.1.2}) is converted to
\begin{equation}
J_\pi(X)= \sum_{t=0}^{\infty} \gamma^t \left( \operatorname{Tr}\left[H\mathcal{M}(X_t)\right] + const'. \right),
\end{equation}
where $H = \operatorname{diag}(Q,R)$ and 
$$
const'. = -\frac{\tau}{2-q}\left(\log_q\left(Z_q\right) + \alpha n +  (q -1)\alpha n \log_q\left(Z_q\right) +1 \right).
$$
To formulate our policy iteration method, we examine the Bellman operator \cite{bertsekas2022abstract}, defined in accordance with the framework presented therein, as
\begin{equation}
\left(\mathcal{T}_\pi J\right)(X)= \operatorname{Tr}\left[\mathcal{M}(X)H\right] + const'.+ \gamma J\left(\mathcal{E}\left(\mathcal{M}(X)\right)\right),
\label{eq2.16}
\end{equation}
and $\mathcal{T} J=\min _\pi \mathcal{T}_\pi J$, operating on $J: \mathcal{X} \rightarrow \mathbb{R}$.
The optimal value of (\ref{eq2.11}) is determined as the fixed-point of $\mathcal{T}$ if the dynamics can be stabilized, meaning that a mean square stabilizing controller exists. The optimal controller, according to (\ref{eq2.16}), is $\operatorname{argmin}_\pi \mathcal{T}_\pi J^*$, where $J^*=\mathcal{T} J^*$.

We aim to determine $K^*$ via policy iteration, which involves the iterative updating of $K$ to identify a $K^+$ such that
\begin{equation}
K^+=\underset{K^{\prime}}{\operatorname{argmin}} \mathcal{T}_{K^{\prime}} J_K, \mbox{ with } J_K=\mathcal{T}_K J_K.
\label{eq2.17}
\end{equation}

\subsection{Q-Functions}
Various methods are available for addressing these equations through making use of data.   Nevertheless, the computation of the minimizer in (\ref{eq2.17}) needs the dynamics, even after determining $J_K$. Therefore, we utilize Q-functions.

A Q-function in this context maps $Z$ to some real number. Following the formulation in \cite{bertsekas2022abstract}, we define the Bellman operator associated with Q-functions as:
\begin{equation}
\left(\mathcal{F}_\pi \mathcal{Q}\right)(Z) = \operatorname{Tr}\left[ZH\right] + const'.+ \gamma \mathcal{Q}_\pi \left(\mathcal{M}\left(\mathcal{E}(Z) \right)\right),
\label{eq2.19}
\end{equation}
and $\mathcal{F} \mathcal{Q}=\min _\pi \mathcal{F}_\pi \mathcal{Q}$. For some policy $\pi$ we define $\mathcal{Q}_\pi$ as
\begin{equation}
\mathcal{Q}_\pi(Z) = \operatorname{Tr}\left[ZH\right] + const'.+\gamma  J_\pi\left(\mathcal{E}(Z)\right) ,
\label{eq2.20}
\end{equation}
where $J_\pi$ solves $J_\pi=\mathcal{T}_\pi J_\pi$. 

\subsection{Tsallis Policy Evaluation}
\begin{theorem}[Tsallis Policy Evaluation]
For fixed $\pi$ and $0<q<1$, consider the operator $\mathcal{F}_{\pi}$, and define Tsallis policy evaluation as $\mathcal{Q}_\pi= \mathcal{F}_\pi \mathcal{Q}_\pi$ for an arbitrary initial guess. Then, $\mathcal{Q}_\pi$ converges to $\mathcal{Q}$ and satisfies the policy evaluation equation (\ref{eq2.19})-(\ref{eq2.20}).
\label{TE-thm-5.2}
\end{theorem}

\begin{proof}[\textbf{Proof of Theorem \ref{TE-thm-5.2}}]
We can show that this choice of $\mathcal{Q}_\pi$ is a fixed-point of $\mathcal{F}_\pi$ :
$$
\begin{aligned}
\left(\mathcal{F}_\pi \mathcal{Q}_\pi\right)(Z) & = \operatorname{Tr}\left[ZH\right] + const'.+ \gamma \left[ \operatorname{Tr}\left(\mathcal{M}\left(\mathcal{E}(Z) \right)H\right)  \right. \\
&\quad \left. + const'. +\gamma J_\pi(\mathcal{E}\left(\mathcal{M}\left(\mathcal{E}(Z) \right)\right))\right]\\
& = \operatorname{Tr}\left[ZH\right] + const'.+ \gamma \left(\mathcal{T}_\pi J_\pi\right)\left(\mathcal{E}(Z)\right) \\
& = \operatorname{Tr}\left[ZH\right] + const'.+ \gamma  J_\pi\left(\mathcal{E}(Z)\right)\\
& = \mathcal{Q}_\pi(Z),
\end{aligned}
%\label{eq2.21}
$$
where we utilized the definition of $\mathcal{T}_\pi$ for the second equality and applied $J_\pi=\mathcal{T}_\pi J_\pi$ for the third.
\end{proof}
The value function evaluated from Tsallis policy evaluation can be employed to update the policy distribution. In the policy improvement step, given such a $\mathcal{Q}_K$ we implement policy iteration as:
\begin{equation}
K^+(x)=\underset{K^{\prime}}{\operatorname{argmin}} \mathcal{Q}_K\left(x\right),
\label{original-K+}
\end{equation}
which is equivalent to $K^{+}=\operatorname{argmin}_{K^{\prime}} \mathcal{F}_{K^{\prime}} \mathcal{Q}_K$.  According to Theorem \ref{thm2.4}, $J_\pi\left(\mathcal{E}(Z)\right)$ in (\ref{eq2.20}) can be written as
\begin{equation}
\begin{aligned} 
&\quad \ J_\pi\left(\mathcal{E}(Z)\right) \\
&= \mathbb{E}_\pi\left[x_+^{\top} P x_+\right]+c\\ 
&=\mathbb{E}_\pi \Bigl[\left((A+\Delta A) x + (B+\Delta B) u+(C x + D u)w\right)^{\top} P   \\
& \quad \cdot ((A+\Delta A) x +(B+\Delta B) u+(C x + D u)w)\Bigr]+c\\
&= \mathbb{E}_\pi \left[x^{\top} (A^{\top} P A + \Sigma_A P + WC^{\top}PC) x + u^{\top}\left(B^{\top} P B +\Sigma_B P \right. \right. \\
&\quad \left. \left. + WD^{\top} P D\right) u+ 2 u^{\top} \left(B^{\top} P A + WD^{\top} PC\right) x\right]+c \\
& = \operatorname{Tr}\left[\Theta'Z \right]+c, \\
\end{aligned}
\label{eq2.22}
\end{equation}
where 
$$
\Theta'  = \begin{bmatrix}
A^\top P A+ \Sigma_AP + WC^\top P C & \gamma A^\top P B + WC^\top PD \\  B^T P A + WD^\top P C & B^\top P B+ \Sigma_B P +  WD^\top P D
\end{bmatrix}.
$$ 
When taking (\ref{eq2.22}) into (\ref{eq2.20}), we can use a linear parametrization \cite{coppens2022policy,wang2018stochastic} regarding $\Theta$ to express $\mathcal{Q}_\pi(Z)$ as
\begin{equation}
\mathcal{Q}_\pi(Z) = \operatorname{Tr}[ZH] + const'.+ \gamma \left(\operatorname{Tr}\left[\Theta'Z \right]+c \right) = \operatorname{Tr}\left[\Theta Z \right]+ const.
\label{eq2.24}
\end{equation}
where 
$$
const. = const'. + \gamma c,
$$ 
and $\Theta = H + \gamma \Theta'$, which is defined as
\begin{equation}
\Theta=
\begin{bmatrix}
\Theta_{x x} & \Theta_{x u} \\
\Theta_{u x} & \Theta_{u u}
\end{bmatrix}.
\label{eq2.23}
%\end{aligned}
\end{equation}
where $\Theta_{x x}  = Q+ \gamma A^\top P A+ \gamma \Sigma_AP + \gamma WC^\top P C $, $\Theta_{x u} =  \gamma A^\top P B + \gamma WC^\top PD$, $\Theta_{u x} = \gamma B^T P A + \gamma WD^\top P C$, $\Theta_{u u} = R+\gamma B^\top P B+ \gamma \Sigma_B P + \gamma WD^\top P D$. 
Then we have 
\begin{equation}
K^{+}=\left(\Theta_\pi\right)_{u u}^{-1}\left(\Theta_\pi\right)_{u x}.
\label{eqK}
\end{equation}
The comparison with $K^{*}$ from the previous section further supports that the optimal gain can be attained by choosing the appropriate value for $\Theta_\pi$.

\subsection{Tsallis Policy Improvement and Convergence}
\begin{theorem}[Tsallis Policy Improvement]
For $0\leq q<1$, let $K^{+}$ be the updated policy from (\ref{eqK}) using $\mathcal{Q}_{K}$. For all $(x, u) \in \mathcal{X} \times \mathcal{U}$, $\mathcal{Q}_{K^{+}}(x)$ is less than or equal to $\mathcal{Q}_{K}(x,u)$.
\label{TE-thm-5.3}
\end{theorem}

\begin{proof}[\textbf{Proof of Theorem \ref{TE-thm-5.3}}]
Since $K^+$ is updated by equation  and the maximization in equation (\ref{original-K+}) and the minimization in (\ref{original-K+}) is convex, the following inequality holds
$$
\begin{aligned}
&\quad \ \mathbb{E}_{\pi}\left[\mathcal{Q}_{K^+}(x,u)-\tau \mathcal{H}_q(\pi_{K^+}(u|x)) | x \right] \\
& \leq \mathbb{E}_{\pi}\left[\mathcal{Q}_{K}(x,u)-\tau \mathcal{H}_q(\pi_{K}(u|x)) | x \right]={J}_{\pi_K}(x),
\end{aligned}
$$
where the equality holds when $\pi_{K^+}=\pi_{K}$. This inequality induces a performance improvement,
$$
\begin{aligned}
\mathcal{Q}_{K}(x,u) &= \mathbb{E}_{\pi}\left[x_0^\top Q x_0 + u_0^\top R u_0 - \tau \mathcal{H}_q(\pi_{K}(u_0|x_0)) \right. \\
&\quad \left. + \gamma {J}_{\pi_K}(x_1) | x_0 = x \right] \\
&\geq \mathbb{E}_{\pi}\left[x_0^\top Q x_0 + u_0^\top R u_0 - \tau \mathcal{H}_q(\pi_{K}(u_0|x_0)) | x_0=x\right] \\
&\quad + \gamma \mathbb{E}_{\pi} \left[x_1^\top Q x_1 + u_1^\top R u_1 - \tau \mathcal{H}_q(\pi_{K^+}(u_1|x_1)) \right. \\
&\quad \left. + \gamma J_K(x_2)  | x_0 = x \right] \\
&\geq \mathbb{E}_{\pi}\left[x_0^\top Q x_0 + u_0^\top R u_0 - \tau \mathcal{H}_q(\pi_{K}(u_0|x_0)) | x_0=x\right] \\
&\quad +\gamma \mathbb{E}_{\pi} \left[ \sum_{k=1}^{t} \gamma^{k-1} c\left(x_k, u_k \right) \Bigg| 
x_0=x \right] \\
&\quad + \gamma^{t+1} \mathbb{E}_{\pi}\left[{J}_{\pi_k}(x_{t+1})  | x_0=x \right] \\
&\vdots \\
&\geq \mathbb{E}_{\pi} \left[  x_0^\top Q x_0 + u_0^\top R u_0 - \tau \mathcal{H}_q(\pi_{K}(u_0|x_0)) | x_0=x \right] \\
&\quad + \gamma \mathbb{E}_{\pi} \left[ {J}_{\pi_{K^+}}(x_1) \right] \\ 
& = \mathcal{Q}_{K^+}(x,u),
\end{aligned}
$$
where $\gamma^{t+1}\mathbb{E}_{\pi}\left[J_{\pi_{K}}(x_{t+1}) | x_0=x \right]  \rightarrow 0$ when $t  \rightarrow0$.
\end{proof}
For the optimality of Tsallis entropy policy iteration, when $0<q<1$, we define the Tsallis policy iteration in an alternative way by applying (\ref{eq2.20}) and (\ref{original-K+}), then $\pi_K$ converges to the optimal policy.

\section{Data-Driven Implementation} \label{secion 3}
\subsection{Data-Related Matrices Construction with Least-Squares}
It is clear from (\ref{eqK}) that the optimal policy depends exclusively on the matrix $\Theta$. thereby avoiding any constraints imposed by the system parameters. The $\Theta$ matrix clearly encompasses the information regarding system dynamics. We choose to estimate $\mathcal{Q}_\pi$ directly from the data. We commence with an examination of data constraints, followed by the derivation of a model equation utilized for a least-squares estimator.

We assume access to samples of subsequent moments $\left\{\left(Z_i,X_{i+}\right)\right\}_{i=1}^N$, where
\begin{equation}
\begin{aligned}
X_{i+}= \begin{bmatrix}A+\Delta A_i+Cw_i & B+\Delta B_i+Dw_i\end{bmatrix}Z_i\\\begin{bmatrix}A+\Delta A_i+Cw_i & B+\Delta B_i+Dw_i\end{bmatrix}^\top.
\label{eqZ_i}
\end{aligned}
\end{equation}
Data can be generated by selecting certain $x_i$ and evaluating $x_{i+}$ using (\ref{eqp-2.1.3}). Subsequently, allow 
$$
Z_i=\begin{bmatrix}x_i \\ u_i\end{bmatrix}\cdot \begin{bmatrix}x_i^\top & u_i^\top\end{bmatrix} \mbox{ and } Z_{i+}=\begin{bmatrix}x_{i+} \\ u_{i+}\end{bmatrix}\cdot \begin{bmatrix}x_{i+}^\top & u_{i+}^\top\end{bmatrix}. 
$$
A comprehensive data-generation procedure is presented in Section \ref{secion 3}. We want to find $\mathcal{Q}_\pi$ such that $\mathcal{Q}_\pi(Z)=\left(\mathcal{F}_\pi \mathcal{Q}_\pi\right)(Z)$ for all $Z$. We sample $\mathcal{F}^i_\pi \mathcal{Q}_\pi$ by: 
\begin{equation}
\left(\mathcal{F}_\pi^i \mathcal{Q}_\pi\right)\left(Z_i\right)= \operatorname{Tr}\left(Z_iH\right) + const'. +  \gamma \mathcal{Q}_\pi(Z_{i+}),
\label{eq2.26}
\end{equation}
which we can use to construct a model equation:
$$
\mathcal{Q}_\pi\left(Z_i\right)=\left(\mathcal{F}_\pi^i \mathcal{Q}_\pi\right)\left(Z_i\right)+\left(\left(\mathcal{F}_\pi \mathcal{Q}_\pi\right)\left(Z_i\right)-\left(\mathcal{F}_\pi^i \mathcal{Q}_\pi\right)\left(Z_i\right)\right).
$$
Using the parametrization in (\ref{eq2.24}) 
results in:
\begin{equation}
\begin{aligned}
\operatorname{Tr}[\Theta_\pi Z_i] + const. &= \operatorname{Tr}[Z_iH] +  const'. + \gamma \operatorname{Tr}[\Theta_\pi Z_{i+}] \\
& \quad + \gamma \operatorname{Tr}[\Theta_\pi \left(\mathcal{M}\left(\mathcal{E}(Z_i) \right) - Z_{i+}\right)],
\end{aligned}
\label{eqLS}
\end{equation}
which is linear in the parameter $\Theta_\pi$ with zero-mean error. It is important to observe that (\ref{eqLS}) incorporates temporal differences, as discussed in \cite{bradtke1996linear}.

We reformulate the model utilizing the symmetric vectorization operator  $\operatorname{vec}_{\mathrm{s}}: \mathbb{R}^{m\times m} \rightarrow \mathbb{R}^{\frac{m(m+1)}{2}}$, which satisfies $\operatorname{Tr}[X Y]=\operatorname{vec}_{\mathrm{s}}(X)^{\top} \operatorname{vec}_{\mathrm{s}}(Y)$. We rewrite (\ref{eqLS}) as:
\begin{equation}
b_i=\left(a_i+e_i\right)^{\top} \theta_\pi
\label{eq2.31}
\end{equation}
with 
\begin{equation}
\begin{aligned}
\theta_\pi&=\operatorname{vec}_{\mathrm{s}}\left(\Theta_\pi\right), \\
b_i&= \operatorname{Tr}(HZ_i) + const'. - const., \\
a_i&=\operatorname{vec}_{\mathrm{s}}\left(Z_i - \gamma Z_{i+}\right),\\
e_i &= \gamma \operatorname{vec}_{\mathbf{s}}\left(Z_{i+} - \mathcal{M}\left( \mathcal{E}(Z_i) \right))\right).
\end{aligned}
\label{eq2.32}
\end{equation}
This is referred to as the error-in-variables setting. The issue arises from the linear dependence of both $a_i$ and $e_i$ on $w_i w_i^{\top}$ and $w_i$, as expressed through $Z_{i+}$. Therefore, $\mathbb{E}\left[a_i e_i^{\top}\right] \neq 0$, resulting in an inconsistency in the classical least-squares estimate.

This stands for the traditional rationale for the use of instrumental variables \cite{young2011recursive}. We employ $g_i:=\operatorname{vec}_{\mathbf{s}}(Z_i)$ as instrumental variables, following the approach outlined in \cite{bradtke1996linear}. This selection is independent of the error $e_i$ (i.e., $\mathbb{E}\left[g_i e_i^{\top}\right]=0$), while exhibiting correlation with $a_i+e_i$, which is a requisite characteristic of an instrumental variable. Conceptually, $g_i$ represents the optimal estimate in the absence of supplementary information. 

The estimate using the instrumental variables is
\begin{equation}
\hat{\theta}_\pi=\left(\sum_{i=1}^N g_i a_i^{\top}\right)^{-1}\left(\sum_{i=1}^N g_i b_i\right).
\label{eq2.33}
\end{equation}
Alternatively we can use the recursive algorithm:
\begin{equation}
\begin{aligned}
\hat{\theta}_{\pi, i} & =\hat{\theta}_{\pi, i-1}+L_i\left[b_i-a_i^{\top} \hat{\theta}_{\pi, i-1}\right], \\
L_i & =S_{\pi, i-1} g_i /\left(1+a_i^{\top} S_{\pi, i-1} g_i\right), \\
S_{\pi, i} & =S_{\pi, i-1}-\left(L_i a_i^{\top}\right) S_{\pi, i-1},
\end{aligned}
\label{eq2.34}
\end{equation}
for $i=1, \cdots, N$ and where the initialization $\hat{\theta}_{\pi, 0} \in \mathbb{R}^{\frac{(m+n)(m+n+1)}{2}}$ and $S_{\pi, 0} \in \mathbb{R}^{\frac{(m+n)(m+n+1)}{2} \times \frac{(m+n)(m+n+1)}{2}}$ represent an initial estimate for $\theta_\pi$ associated confidence in that estimate, respectively.

\subsection{Dataset Persistent Excitation (PE) condition}\label{sec3.1}
The process of data generation is outlined as follows. To derive the algorithms, we regard them as iteratively updating the policy. Each iteration employs $M$ trajectories of length $T$ for a total of $N= MT$ samples, yielding enough new data points that fulfill (\ref{eqZ_i}). Parallel simulations are conducted for $j=1, \cdots, M$ and $t=0, \cdots, T$ as $T \rightarrow \infty$:
$$
x_{t+1}^j=(A+\Delta A_t^j) x_t^j+(B+\Delta B_t^j) u_t^j+ (C x_t^j + D u_t^j) w_t^j
$$
and let $z_t^j=\begin{bmatrix}x_t^j \\ u_t^j\end{bmatrix}$ for
\begin{equation}
u_t^j=-K x_t^j+\nu_t^j, \quad \nu \sim \mathcal{N}_q(0, \Sigma),
\label{eq3.1}
\end{equation}
where $\nu_t^j$ is uniformly distributed over $\left\{\nu:\|\nu\|_2 \leq r_\nu\right\}$. We sample $\nu_t^j$  from a ball to prevent over-excitation of the system, while ensuring that the trajectories remain sufficiently informative.

The functional form of the dynamics is unnecessary as only states and inputs are collected. A simulator or experiments are sufficient. Each trajectory must be initialized at the beginning of a new controller iteration. We can either proceed with trajectories from the prior iteration, or we can sample $\left\{x_0^j\right\}_{j=1}^M$ uniformly from the set $\left\{x:\|x\|_2 \leq r_x\right\}$. We estimate the snapshots \(X_t^i\) and \(Z_t^i\) for each augmented state vector through the application of an outer product. In cases where the specific trajectory does not influence the identification procedure, such as in policy iteration and system identification, we denote the pairs $\left(Z_t^j, X_{t+1}^j\right)$ using the notation $\left(Z_i, X_{i+}\right)$, with $t$ and $j$ varying within their respective ranges. All these pairs satisfy (\ref{eqZ_i}).

\subsection{Policy Iteration Algorithm}
We can now explain the complete approximate policy iteration algorithm. The term "generate pair" pertains to the rollouts outlined in the preceding section.
\begin{algorithm}
\caption{Approximate Policy Iteration}
\textbf{Input:} Initial guess $\hat{\theta}_{\pi_0}$, number of samples $N$, and confidence parameter $\beta_0$. \\
\textbf{Initialize:} $S_{\pi_0} = \beta_0 I$.
\begin{algorithmic}[1]
\For{$t = 1$ to $\infty$}
    \State $K_t = -\Theta_{uu}^{-1} \Theta_{ux}$, with $\Theta = \text{unvec}_\text{s}\left(\hat{\theta}_{\pi_{t-1}}\right)$.
    \State Let $\hat{\theta}_{\pi_t, 0} = \hat{\theta}_{\pi_{t-1}}$ and $S_{\pi_t, 0} = S_{\pi_{t-1}}$.
    \For{$i = 1$ to $N$}
        \State Generate pair $\left(Z_i, X_{i+}\right)$ and find $Z_{i+}$. 
        \State Update $\hat{\theta}_{\pi_t, i}$ and $S_{\pi_t, i}$ using Equation (\ref{eq2.34}).
    \EndFor
    \State Set $\hat{\theta}_{\pi_t} = \hat{\theta}_{\pi_t, N}$ and $S_{\pi_t} = S_{\pi_t, N}$.
\EndFor
\end{algorithmic}
\label{alg1}
\end{algorithm}
Our approach allows the scheme to continuously enhance its estimation of the Q-function, as evidenced by the numerical experiments. 

\section{Numerical Examples}
\subsection{Application: Mean-Variance Portfolio Selection}
Consider an investor who enters the capital market with initial wealth $x_0$ and invests simultaneously in one riskless asset and $k$ risky assets with an finite time horizon. Let $r_t$ be a given deterministic return of the riskless asset at time period $t$ and $e_t=\left(e_t^1, \cdots, e_t^k\right)^\top$ the vector of random returns of the $k$ risky assets at period $t$. We assume that vectors $e_t, t=0,1, \cdots, T$, are statistically independent and the only information known about the random return vector $e_t$ is its first two moments: its mean $\mathbb{E}\left(e_t\right)=\left(\mathbb{E} e_t^1, \mathbb{E} e_t^2, \cdots, \mathbb{E} e_t^k\right)^\top$ and its covariance $\operatorname{Cov}\left(e_t\right)=\mathbb{E}\left[\left(e_t-\mathbb{E} e_t\right)\left(e_t-\mathbb{E} e_t\right)^\top\right]$. Clearly, $\operatorname{Cov}\left(e_t\right)$ is nonnegative definite, i.e., $\operatorname{Cov}\left(e_t\right) \geq 0$.  

Inspired by the foundational work of \cite{markowitz1952portfolio}, and its extensions in \cite{cui2014optimal} and \cite{wu2018explicit}, we study a multi-period mean-variance (MV) portfolio optimization problem over an infinite horizon. In such cases, it is more realistic to evaluate cumulative risk over time, incorporating ongoing risk aversion and sustained portfolio performance, rather than focusing on a terminal outcome. To proceed, let $x_t$ be the wealth of the investor at the beginning of the $t$-th period, and let $u_t^i, i=1,2, \cdots, k$, be the amount invested in the $i$-th risky asset at period $t$, and $x_t$ be the wealth level in period $t$.
\begin{subequations}   
\begin{align}
\mbox { (MV-$\infty_1$) :}\quad &\min_{u_t} 
\sum_{t=0}^\infty \operatorname{Var}  \left[x_t\right]+ \mathbb{E}\left[\sum_{t=0}^\infty u_t^{\top} R u_t\right] \label{main_func_inf} \\
&\mbox { s.t. } \mathbb{E}\left[x_t\right]=d_t, \mbox{ and } \label{target_d_inf}\\
& \qquad \  x_{t+1}=r_t x_t+{P}_t {u}_t,  \label{eqeg.1b_inf}
\end{align}
\label{MV_inf}
\end{subequations}
for $t=0, \cdots, \infty$. Here ${u}_t \triangleq \left(u_t^1, u_t^2, \cdots, u_t^k\right)^{\top}$ and ${P}_t \triangleq \left(P_t^1, P_t^2, \cdots, P_t^k\right)=$ $\left(e_t^1-r_t, e_t^2-r_t, \cdots, e_t^k-r_t\right)$ is the excess return vector. $d_t$ is a predetermined price level for the expected wealth $\mathbb{E}[x_t]$ at time $t$, with $d_t \geq x_0 \prod_{k=0}^{t} r_k$, which requires that the target wealth level must exceed the wealth level obtained from fully investing in the risk-free asset. The variance sum of $x_t$ is minimized subject to the projected total wealth. $u_t^\top R u_t$ serves as a penalty term for investing in the risky asset, indicating the investor's desire to regulate the proportion of their wealth allocated to such assets. 

Consider an Lagrange function by introducing Lagrangian multiplier $2\lambda \in \mathbb{R}$ for (\ref{target_d_inf}):
$$
\begin{aligned} 
&\quad \  \sum_{t=0}^\infty \operatorname{Var}  \left[x_t\right]+2 \lambda\left(\sum_{t=0}^\infty \mathbb{E}\left[x_t\right]-d_t\right)+ \mathbb{E} \left[ \sum_{t=0}^{\infty}  u_t^{\top} R u_t\right] \\ 
&=\mathbb{E}\left[\sum_{t=0}^\infty \left(x_t-d_t\right)^2+2 \lambda\left(x_t-d_t\right) \right]+ \mathbb{E}\left[\sum_{t=0}^{\infty} u_t^{\top} R u_t\right],
\end{aligned}
$$
where $\operatorname{Var}\left[x_t\right] = \mathbb{E}\left[x_t^2\right]-\mathbb{E}\left[x_t\right]^2 = \mathbb{E}\left[\left(x_t-d_t\right)^2\right]$. This further leads to following equivalent problem (MV-$\lambda_1$):
\begin{equation}
\begin{aligned}
\mbox { (MV-$\lambda_1$) :}\quad &\min_{u_t} 
\mathbb{E}\left[ \sum_{t=0}^\infty \left(x_t-\left(d_t-\lambda\right)\right)^2 + u_t^{\top} R u_t\right] -\lambda^2 \\
&\mbox { s.t. } x_{t+1}=r_t x_t+{P}_t {u}_t, \quad t=0, \cdots, \infty.
\end{aligned}
\label{MV-L1}
\end{equation}
The aforementioned auxiliary problem can be reformulated as a separable linear quadratic control problem (MV-$\lambda_2$), by replacing the state variable $x_t$ with $y_t + d_t-\lambda$:
\begin{equation}
\begin{aligned}
\mbox { (MV-$\lambda_2$) :}\quad &\min_{u_t} 
\mathbb{E}\left[ \sum_{t=0}^\infty y_t^2 +  u_t^{\top} R u_t\right]\\
&\mbox { s.t. } y_{t+1}=r_t y_t+{P}_t {u}_t, \quad t=0, \cdots, \infty.
\end{aligned}
\label{MV-L2}
\end{equation}
Time-inconsistency often arises in dynamic decision-making when future preferences misalign with current plans, a phenomenon likened to the “random shoe” behavior of Wall Street \cite{guo2017quantitative}. The MV problem is inherently time-inconsistent due to its nonlinear dependence on expected terminal wealth, which breaks dynamic programming principles \cite{bjork2010general}. While time-consistent strategies can be constructed in finite-horizon settings \cite{ni2017time}, our infinite-horizon analysis adopts a precommitment solution, with results based on simulations rather than real financial data.

\subsection{Model Formulation}
This section studies the infinite-horizon MV problem with Tsallis entropy regularization in the objective function:
\begin{equation}
%\begin{aligned}
\min_{u_t}  \mathbb{E}_\pi\Biggl[\sum_{t=0}^{\infty} \gamma^t \biggl(x_t^2 + u_t^{\top} R u_t  + \frac{\tau}{2-q} \left(\log_q \pi (u_t | x_t) -1\right) \biggr) \Bigg| x_0=x\Biggr].
\label{egeq.add2}
%\end{aligned}
\end{equation}
The discount factor $\gamma$ reflects the time value of money or utility and to remain the total cost finite. Then, we shall transform (\ref{eqeg.1b_inf}) into a linear controlled system of form (\ref{eqp-2.1.3}), by which the general theory in above sections will work. Precisely, we define
$$
\left\{\begin{array}{l}
w_{i, t}=e_t^i-\mathbb{E}\left(e_t^i\right), \\
D_{i, t}=(0, \cdots, 0,1,0, \cdots, 0), \\
i=1, \cdots, k, t=0,1, \cdots, \infty,
\end{array}\right.
$$
which leads to 
\begin{equation}
x_{t+1}=\left(r_t x_t+\mathbb{E}(P_t) u_t\right)+\sum_{i=1}^n D_{i, t} u_t w_{i, t}.
\label{eqeg.2}
\end{equation}
It is not hard to see that this MV problem is just a special case of Section \ref{LQmodel} by letting $x_t \in \mathbb{R}$, $u_t \in \mathbb{R}^n$, $C=0$, $Q = 1 (\in \mathbb{R})$, $R \in \mathbb{R}^{n \times n}$, $A=r_t$, $B=\mathbb{E}(P_t)$ and $\Delta A_t = \Delta B_t = 0$, for $t=0, \cdots$.\\
\textbf{Setup.} (1) Parameters: $n=1$, $k=3$, $q=0.8$ for Tsallis entropy. The matrices $A$, $B$, $R$, and $W$ are defined as follows: $A=1.057$ (set to reflect a risk-free return rate of 5.7\%), $B = \begin{bmatrix} 0.21 & 0.28 & 0.22 \end{bmatrix}$, and $W = 0.99$. $R$ is generated to be a very small positive definite matrix, with each entries from 0 to 0.1. The regularization parameter $\tau$ is set to $0.7$. (2) We begin with a setup optimized for offline policy iteration (PI), where data is generated as prescribed by Algorithm \ref{alg1}. Subsequently, policy iteration is evaluated in an online policy setting.

\begin{figure}[h]
\centering
\includegraphics[width=0.8\textwidth]{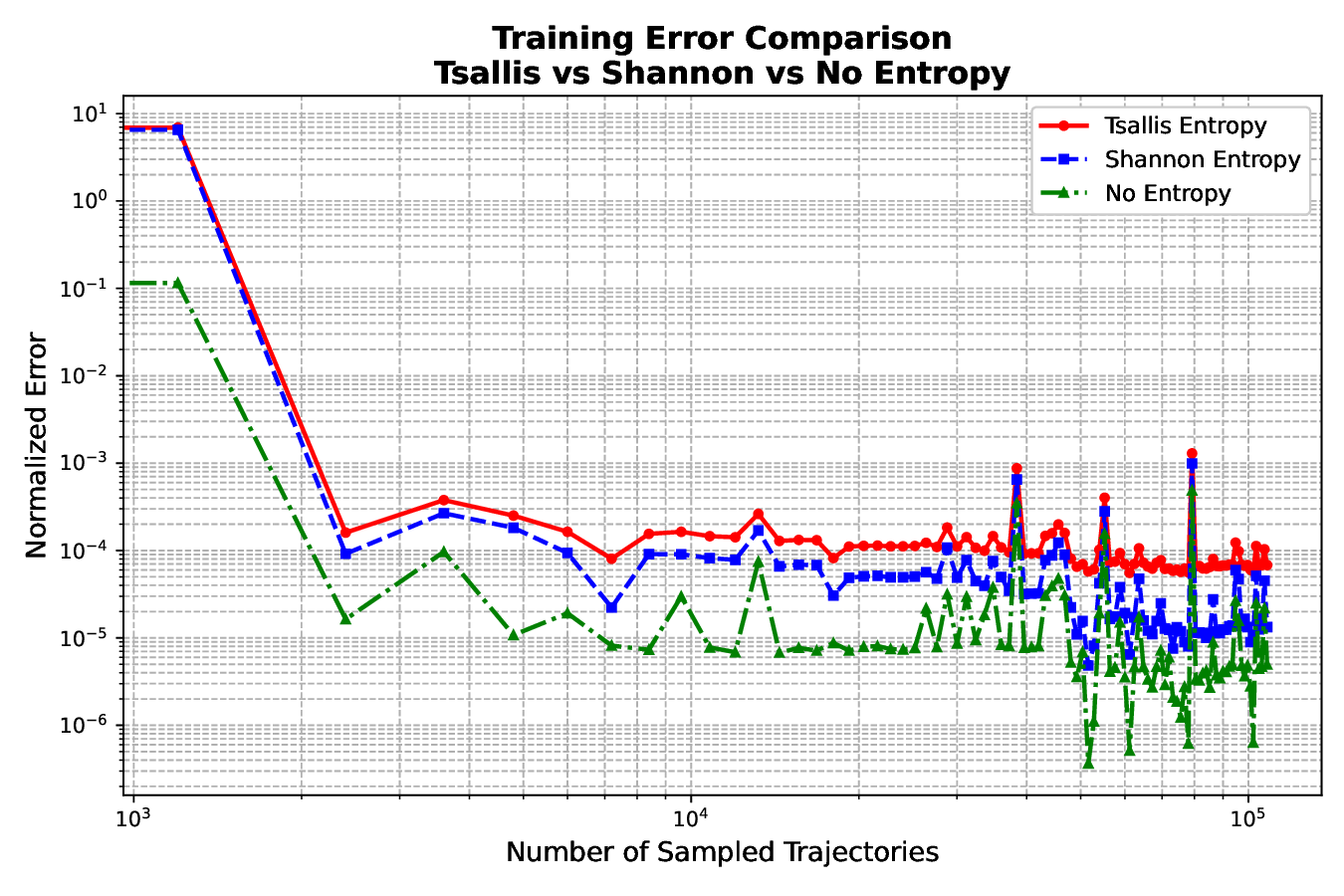}
\caption{Performance of offline policy iteration}
\label{fig1}
\end{figure}

\subsection{Offline and Online Policy Iteration}
We initially examine the influence of the number of rollouts on the effectiveness of our policy iteration method. According to the notation defined in section \ref{sec3.1}, the additive noise radius is set to $r_\nu=0.8$. Each policy iteration update employs $N=1200$ trajectories, initiated from states with a distribution parameter of $r_x=1$. A constant, stabilizing control gain $K_0$ is utilized to generate the samples. We set the initial value of policy iteration as $\hat{\theta}_{\pi_0} \in \mathbb{R}^{(1+k)(1+k+1)/2}$ with all entries are 0.8 and $\beta_0=20$, unless specified differently. The initial Q-function's optimal controller is represented by $K_0$, which ensures stability. 

In Figure \ref{fig1}, we compare the convergence behavior of three policy iteration variants: Tsallis entropy, Shannon entropy, and no entropy. The Tsallis exhibits smoother and more stable convergence, while the no-entropy case suffers from the largest fluctuations, likely due to sensitivity to sampling noise. Although all three methods reach normalized errors within the $10^{-4}$ to $10^{-5}$ range, their convergence dynamics differ: entropy-regularized methods reduce variance at the cost of slower early progress. The Tsallis and Shannon cases, using $q=0.8$ and $q=1$ respectively, converge at similar rates, highlighting the stabilizing role of entropy regularization.

\begin{figure}[htbp]
\centering
\begin{subfigure}[t]{0.45\columnwidth}
    \centering
    \includegraphics[width=\linewidth]{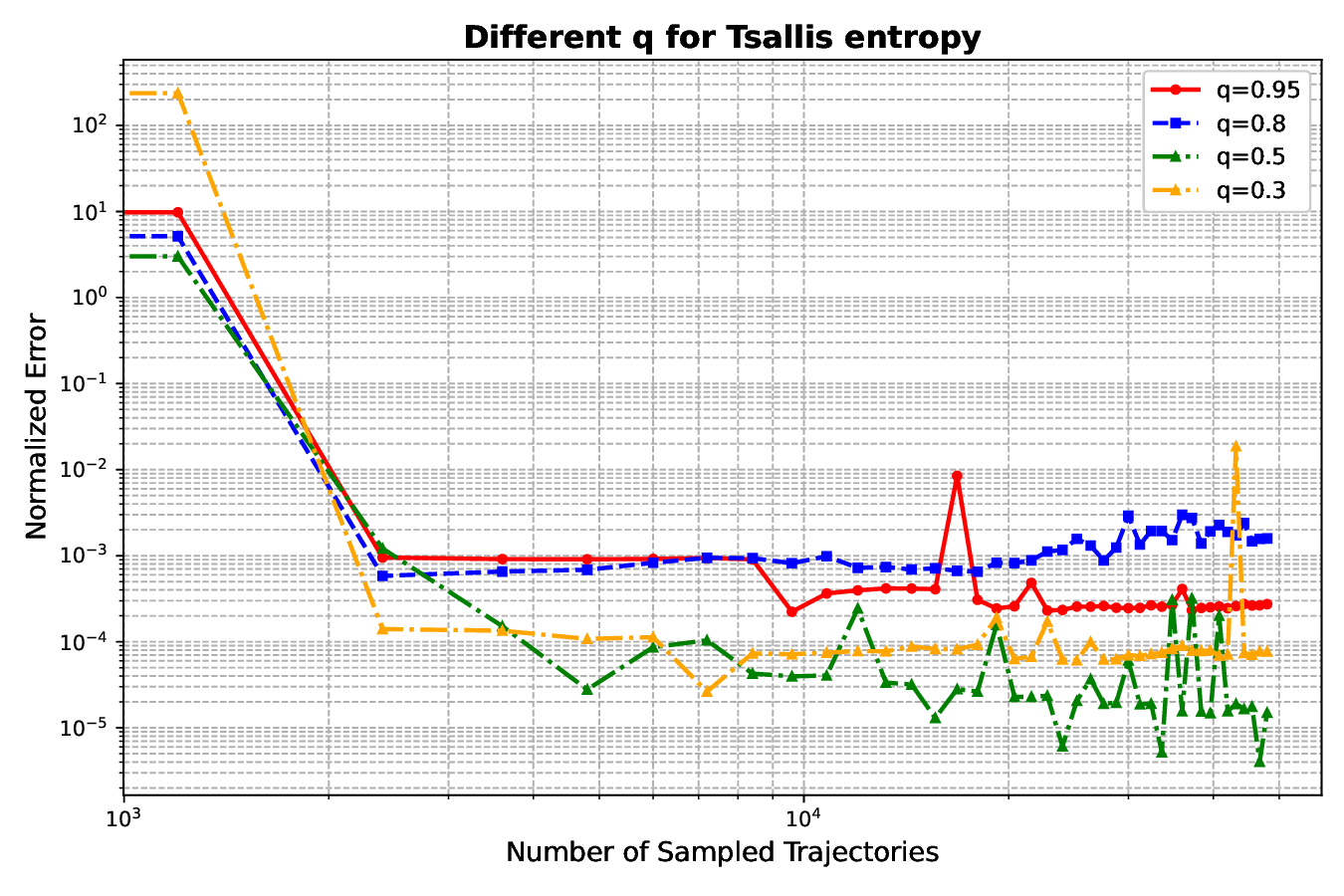}
    \caption{Different $q$ values}
    \label{fig:q_comparison}
\end{subfigure}%
\hspace{0.02\columnwidth}%
\begin{subfigure}[t]{0.45\columnwidth}
    \centering
    \includegraphics[width=\linewidth]{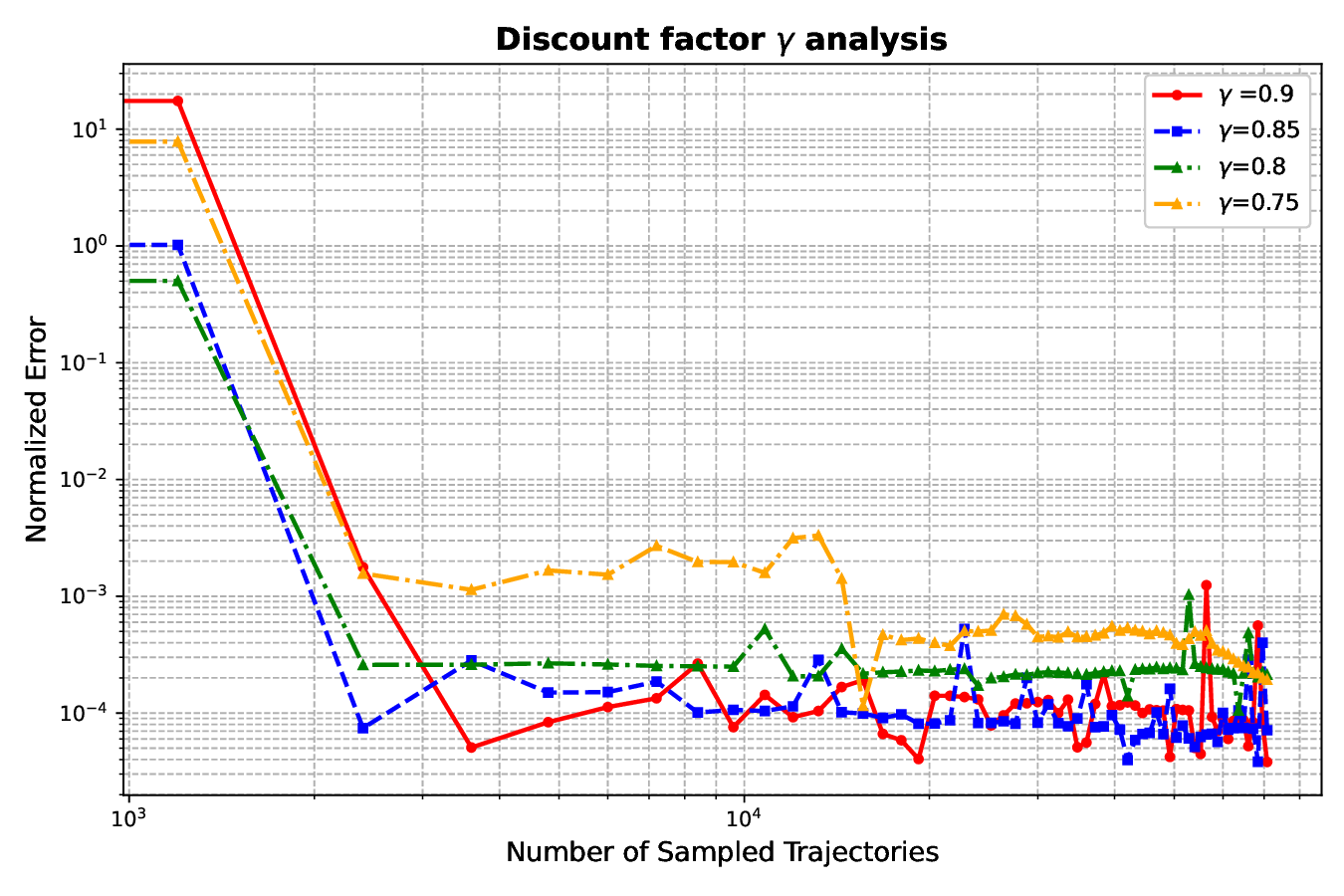}
    \caption{Different $\gamma$ values}
    \label{fig:gamma_comparison}
\end{subfigure}
\caption{Convergence under varying parameters}
\label{fig2}
\end{figure}

Figure \ref{fig:q_comparison} explores the impact of different Tsallis entropy parameters $q \in [0.3, 0.95]$ on convergence. While no strict monotonic relationship emerges, smaller $q$ values tend to accelerate convergence but introduce greater fluctuations. In contrast, larger $q$ values promote smoother but slower convergence. Though our analysis assumes $q<1$, comparing to $q>1$ could provide further insight into the balance between exploration and regularization.

Figure \ref{fig:gamma_comparison} presents convergence under Tsallis entropy with varying discount factors $\gamma=0.75,0.80,0.85,0.90$. Smaller $\gamma$ values lead to faster convergence but higher asymptotic errors, whereas larger $\gamma$ induces slower and more fluctuating learning but eventually achieves lower final error. This illustrates a classic reinforcement learning trade-off: short-term reward bias improves stability and speed, while long-term planning may yield better policies at the cost of convergence smoothness.

\begin{figure}[h]
\centering
\includegraphics[width=0.8\textwidth]{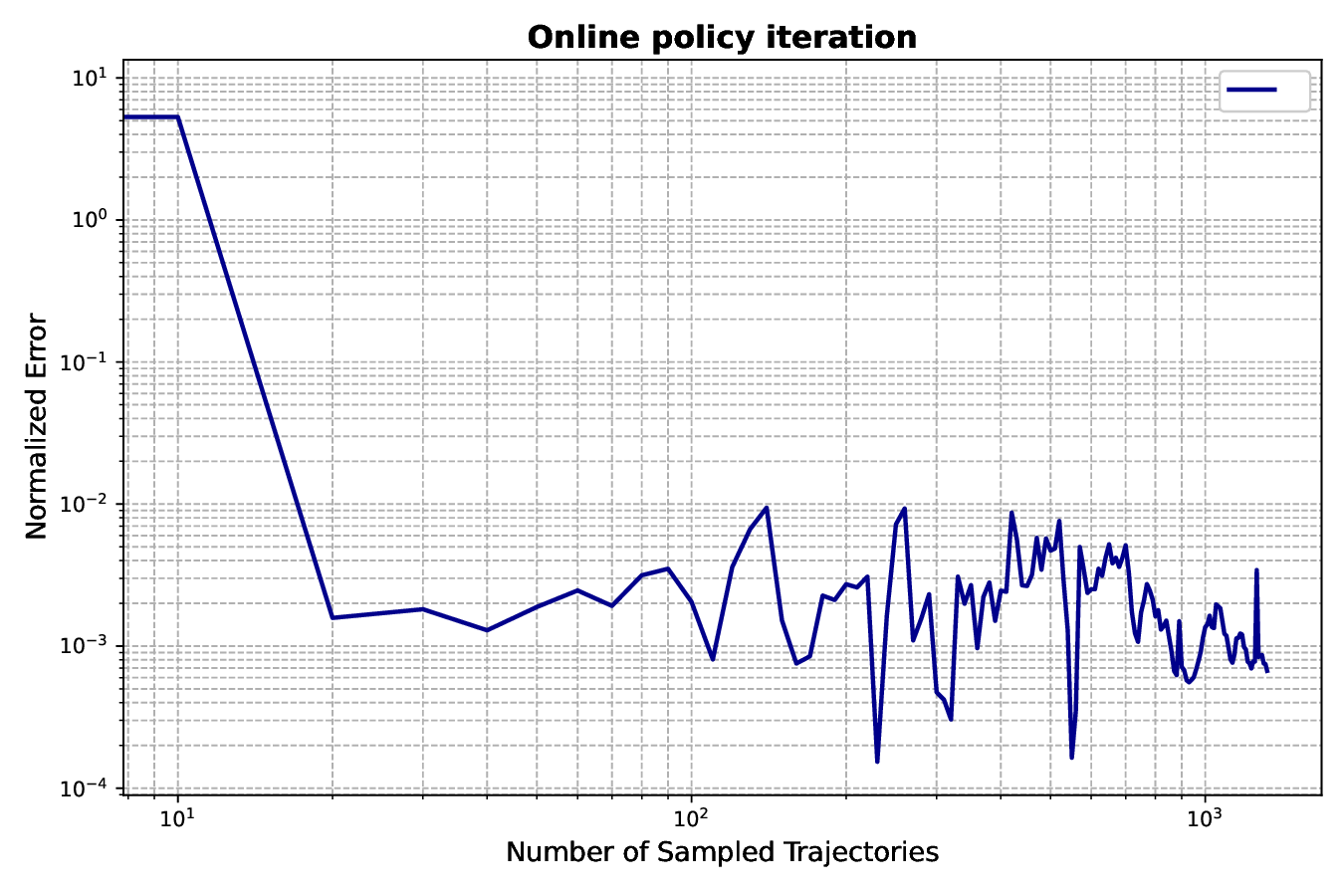}
\caption{Performance of online policy iteration}
\label{fig-online}
\end{figure}

In the online policy iteration setting, parameters are updated in real-time using freshly sampled data per iteration, unlike the offline method which relies on a fixed dataset. Using the same sample size $N$ and initialization $\beta_0$, the online approach reaches the convergence region more quickly due to its adaptive nature, though it experiences greater fluctuation. Nonetheless, both settings achieve comparable final error levels, confirming the robustness of the online scheme despite its increased variance.

\section{Conclusions}
This paper develops a novel reinforcement learning framework for the LQ control problem under multiplicative noise, incorporating Tsallis entropy as a regularization mechanism. By extending classical entropy-regularized control methods to the Tsallis setting, we offer a tunable approach to balance exploration and sparsity in the control strategy. We formulate a policy iteration scheme grounded in Q-functions, derive explicit expressions for the optimal policy, and prove the convergence of the proposed algorithm.

For cases with unknown system dynamics, we propose a fully data-driven variant using instrumental-variable-based least-squares estimation, enabling stable and efficient policy updates without prior knowledge of system parameters. Our method demonstrates both theoretical soundness and practical potential, filling a critical gap in the literature on entropy-regularized stochastic control under multiplicative uncertainty. These results lay the groundwork for future applications of generalized entropy methods in robust, model-free control for high-dimensional and uncertain systems.

\end{document}